\documentclass[%envcountsame%,
envcountsect,
]{llncs}

\usepackage{amsmath,amssymb}
\usepackage{latexsym}
\usepackage{amsfonts} 
\usepackage[dvipdfm]{color,graphicx}
\usepackage[english]{babel}
\usepackage{hyperref}
\usepackage{idba}

\spnewtheorem*{notation}{Notation}{\itshape}{\rmfamily}

\title{Characterising Complexity Classes by Inductive Definitions in
Bounded Arithmetic%
\thanks{This work started as part of 
        \href{http://www.templeton.org/what-we-fund/grants/philosophical-frontiers-in-reverse-mathematics}{\textit{Philosophical
	frontiers in Reverse Mathematics}}
        sponsored by the John Templeton Foundation.}
}

\author{Naohi Eguchi%
\thanks{The author is supported by JSPS posdoctoral fellowships for young
	scientists.%
       }
}
\institute{%
Institute of Computer Science, University of Innsbruck,
Austria \\
\email{naohi.eguchi@uibk.ac.at}
}

\date{June 7, 2013}

\begin{document}

\maketitle

\begin{abstract}
Famous descriptive characterisations of $\mathrm P$ and 
$\mathrm{PSPACE}$ are restated in terms of the Cook-Nguyen style second order bounded arithmetic. 
We introduce an axiom of inductive definitions over second  order bounded arithmetic.
We show that $\mathrm P$ can be captured by the axiom of inflationary
 inductive definitions whereas $\mathrm{PSPACE}$ can be captured by the
 axiom of non-inflationary inductive definitions.
\end{abstract}

\section{Introduction}

The notion of inductive definitions is widely accepted in logic and
mathematics. 
Although inductive definitions usually deal with infinite sets,
we can also discuss finitary inductive definitions.
Let $S$ be a finite set and 
$\Phi: \mathcal P (S) \rightarrow \mathcal P (S)$
an {\em operator}, 
a mapping over the power set $\mathcal P (S)$ of $S$.
For a natural $m$, define a subset $P^m_\Phi$ of $S$ inductively by
$P^0_\Phi = \emptyset$ and
$P^{m+1}_{\Phi} = \Phi (P^m_\Phi)$.
If the operator $\Phi$ is {\em inflationary}, i.e.,
if $X \subseteq \Phi (X)$ holds for any $X \subseteq S$, then
there exists a natural $k \leq |S|$ such that
$P^{k+1}_{\Phi} = P^k_\Phi$, where $|S|$ denotes the number of elements
of $S$, and hence the operator $\Phi$ has a fixed
point.
On the side of finite model theory, %based on this observation, 
a famous descriptive characterisation of the class of $\mathrm{P}$ of
polytime predicates was given by N. Immerman \cite{Immerman82} and
M. Y. Vardi \cite{Vardi82}. 
It is shown that the class $\mathrm{P}$ can be captured by the first order
predicate logic with fixed point predicates of first order definable
inflationary operators.
In case that the operator $\Phi$ is not inflationary, it is not in general
possible to find a fixed point of $\Phi$.
One can however find two naturals $k, l \leq 2^{|S|}$ such that
$l \neq 0$ and $P^{k+l}_{\Phi} = P^k_{\Phi}$.
Based on this observation, it is shown that the class
$\mathrm{PSPACE}$ of polyspace predicates %by Neil Immerman 
can be captured by the first order predicate logic with fixed point predicates
of first order definable (non-inflationary) operators, cf. \cite{FMT99}.
On the side of bounded arithmetic, it was shown by S. Buss
that
$\mathrm{P}$ can be captured by a
first order system $\mathrm S^{1}_2$ whereas
$\mathrm{PSPACE}$ can be captured by a second order extension 
$\mathrm U^{1}_2$ of  $\mathrm S^{1}_2$, cf. \cite{buss86}.
An alternative way to characterise $\mathrm P$ was invented by
D. Zambella \cite{Zam96}.
As well as Buss' characterisation by $\mathrm S^1_2$,  $\mathrm{P}$ can be captured by a
certain form of comprehension axiom over a weak second order system of
bounded arithmetic.
A modern formalisation of Zambella's idea including further discussions
can be found in the book \cite{CN10} by S. Cook and P. Nguyen.
More recently, A. Skelley in \cite{skelley04} extended this idea to a third order
formulation of bounded arithmetic, capturing $\mathrm{PSPACE}$ as well as Buss'
characterisation by $\mathrm U^1_2$.
%Each of $\mV{1}$ and $\mW{1}{1}$ is based on a comprehension axiom (or equivalently an axiom of induction). 
%Shioji and Tanaka \cite{ShiojiT90},
On the other side, as discussed by K. Tanaka \cite{Tanaka89,Tanaka91}
and others, cf. \cite{Poh98}, inductive definitions over infinite sets of naturals can be
axiomatised over second order arithmetic the most elegantly.
All these motivate us to introduce an axiom of inductive definitions
over second order bounded arithmetic.
Let us recall that for each $i \geq 0$ the class $\BSig{i}$ of formulas is
defined in the same way as the class $\mSigma^{1}_{i}$ of second order
formulas, but only {\em bounded quantifiers} are taken into account.
We show that, over a suitable base, system the class $\mathrm{P}$ can be
captured by the axiom of inductive definitions under
$\BSig{0}$-definable inflationary operators 
(Corollary \ref{c:P}) whereas
$\mathrm{PSPACE}$ can be captured by the axiom of inductive
definitions under $\BSig{0}$-definable (non-inflationary) operators
(Corollary \ref{c:PSPACE}).
There is likely no direct connection, but this work is also partially
motivated by the axiom AID of Alogtime inductive definitions introduced
by T. Arai in \cite{Arai00}.

After the preliminary section, in Section \ref{s:AID} we introduce a system 
$\BSig{0} \text{-IID}$ of inductive definitions under
$\BSig{0}$-definable inflationary operators and a system
$\BSig{0} \text{-ID}$ of inductive definitions under
$\BSig{0}$-definable (non-inflationary) operators. 
In Section \ref{s:P} we show that every polytime function can be defined
in $\BSig{0} \text{-IID}$.
In Section \ref{s:IID} we show that conversely the system
$\BSig{0} \text{-IID}$ can only define polytime functions
by reducing $\BSig{0} \text{-IID}$ to Zambella's system $\mV{1}$.
In Section \ref{s:PSPACE} we show that every polyspace function can be defined
in $\BSig{0} \text{-ID}$.
In Section \ref{s:ID} we show that conversely the system
$\BSig{0} \text{-ID}$ can only define polyspace functions
by reducing $\BSig{0} \text{-ID}$ to Skelley's system $\mW{1}{1}$.

\section{Preliminaries}

The {\em two-sorted first order vocabulary} $\LA{2}$ consists of 
$0$, $1$, $+$, $\cdot$, $|~|$, $=_1$, $=_2$, $\leq$ and $\in$.
At the risk of confusion, we also call $\LA{2}$ the {\em second order vocabulary
of bounded arithmetic}.
Note that $=_1$ and $=_2$ respectively denote the first order and the
second order equality, and $t =_1 s$ or $U =_2 V$ will be simply written
as $t=s$ or $U=V$.
%Note however that the second order equality $U =_2 V$ can be replaced with
%$|U| = |V| \wedge 
% (\forall i < |U|) U(i) \leftrightarrow V(i)$.
First order elements $x, y, z, \dots$ denote natural numbers whereas
seconder order elements $X, Y, Z, \dots$ denote binary strings.
The formula of the form $t \in X$ is abbreviated as $X(t)$.
Under a standard interpretation, $|X|$ denotes the length of the string
$X$, and $X(i)$ holds if and only if the $i$th bit of $X$ is $1$.
Let $\mathcal L$ be a vocabulary  such that
$\mathcal{L}^2_A \subseteq \mathcal L$.
We follow a convention that for an $\mathcal L$-term $t$,
a string variable $X$ and a formula $\varphi$,
$(\exists X \leq t) \varphi$ stands for 
$\exists X (|X| \leq t \wedge \varphi)$ and
$(\forall X \leq t) \varphi$ stands for 
$\forall X (|X| \leq t \rightarrow \varphi)$.
Furthermore $(\exists \vec x \leq \vec t) \varphi$ stands for
$(\exists x_1 \leq t_1) \cdots (\exists x_k \leq t_k) \varphi$
if 
$\vec x = x_1, \dots, x_k$ and
$\vec t = t_1, \dots, t_k$. 
We follow similar conventions for 
$(\forall \vec x \leq \vec t) \varphi$,
$(\exists \vec X \leq \vec t) \varphi$
and
$(\forall \vec X \leq \vec t) \varphi$.
A quantifier of the form $(Q x \leq t)$ or $(Q X \leq t)$ is called a
{\em bounded quantifier}.
Specific classes $\BSig{i} (\mathcal L)$ and 
$\BPi{i} (\mathcal L)$
($0 \leq i$) are defined by the following clauses.

\begin{enumerate}
\item $\BSig{0} (\mathcal L) = \BPi{0} (\mathcal L)$ 
      is the set of $\mathcal L$-formulas whose quantifiers are bounded
      number ones only.
\item $\BSig{i+1} (\mathcal L)$ 
      ($\BPi{i+1} (\mathcal L)$ resp.) is the set of formulas
      of the form $(\exists \vec{X} \leq \vec{t}) \varphi (\vec{X})$
      ($(\forall \vec X \leq \vec t) \varphi (\vec X)$ resp.),
      where $\varphi$ is a $\BPi{i} (\mathcal L)$-formula
      (a $\BSig{i} (\mathcal L)$-formula resp.) and
      $\vec t$ is a sequence of $\mathcal L$-terms not involving any
      variables from $\vec{X}$.
\end{enumerate}
Finally the class $\BDel{i} (\mathcal L)$ is defined in the most natural way for
each $i \geq 0$.
We simply write $\BSig{i}$ ($\BPi{i}$ resp.) to denote
$\BSig{i} (\LA{2})$ 
($\BPi{i} (\LA{2})$ resp.)
if no confusion likely arises.
Let us recall that for each $i \geq 0$ the system $\mV{i}$ is
axiomatised over $\LA{2}$ by the defining axioms for numerical and string function
symbols in $\LA{2}$ (B1--B12, L1, L2 and SE, see \cite[p. 96]{CN10}) and
the axiom $(\BSig{i} \text{-COMP})$ of comprehension for 
$\BSig{i}$ formulas:
\begin{equation}
\tag{$\BSig{i} \text{-} \mathrm{COMP}$}
\forall x (\exists Y \leq x) (\forall i < x)
[Y(i) \leftrightarrow \varphi (i)],
\end{equation}
where $\varphi \in \BSig{i}$.
We will use the following fact frequently.

\begin{proposition}
[Zambella \cite{Zam96}]
\label{p:comp_ind}
{\normalfont (Cf. \cite[p. 98, Corollary V.1.8]{CN10})}
The axiom $(\BSig{i} \text{-} \mathrm{IND})$ of induction for 
$\BSig{i}$ formulas holds in $\mV{i}$.
\end{proposition}

Let $\LA2 \subseteq \mathcal L$.
For a string function $f$, a class $\Phi$ of $\mathcal L$-formulas and a
system $T$ over $\mathcal L$,
we say $f$ is $\Phi$-definable in $T$ if there exists an
$\mathcal L$-formula
$\varphi (\vec X, Y) \in \Phi$ such that
\begin{itemize} 
\item $\varphi$ does not involve free variables other than 
      $\vec X$ nor $Y$,
\item the graph $f(\vec X) = Y$ of $f$ is expressed by
      $\varphi (\vec X, Y)$ under a standard interpretation as mentioned
      at the beginning of this section, and
\item the sentence
      $\forall \vec X \exists ! Y \varphi (\vec X, Y)$
      is provable in $T$.
\end{itemize}
\label{d:definable}
Note that every function over natural numbers can be regarded as a string
one by representing naturals in their binary expansion.

\begin{proposition}[Zambella \cite{Zam96}]
\label{p:ptime}
{\normalfont (Cf. \cite[p. 135, Theorem VI.2.2]{CN10})}
A function is polytime computable if and only if it is 
$\BSig{1}$-definable in $\mV{1}$.
\end{proposition}

\section{Axiom of Inductive Definitions}
\label{s:AID}

In this section we introduce an axiom of inductive definitions.
We work over a conservative extension of $\mV{0}$.
For the sake of readers' convenience, from Cook-Nguyen \cite{CN10}, we recall
several string functions, all of which have $\BSig{0}$-definable bit-graphs.
Let
$\numseq{x,y} = (x+y) (x+y+1) +2y$
be a standard numerical paring function.
Clearly the paring function is definable in $\LA{2}$.
%The paring function can be generalised to a tuple function 
%$\numseq{x_1, \dots, x_{k+1}}$ by
%$\numseq{x_1, \dots, x_{k+1}} = 
% \numseq{\numseq{x_1, \dots, x_k}, x_{k+1}}$.

({\em String encoding} \cite[p. 114 Definition V.4.26]{CN10}) 
The $x$th component $Z^{[x]}$ of a string $Z$ is defined by the axiom
$
 Z^{[x]} (i) \leftrightarrow i< |Z| \wedge Z(\numseq{x, i})
$.

({\em Encoding of bounded number sequences}
\cite[p. 115 Definition V.4.31]{CN10}) 
The $x$th element $(Z)^{x}$ of the sequence encoded by $Z$ is defined by
the axiom
\begin{equation*}
\begin{array}{rclc}
 (Z)^{x} = y & 
 \leftrightarrow & 
 [y < |Z| \wedge Z(\numseq{x, y}) \wedge 
  (\forall z < y) \neg Z(\numseq{x, z})
 ] 
 &
\vee 
\\ &&
 [y=|Z| \wedge (\forall z < y) \neg Z(\numseq{x, z})
 ]. 
 &
\end{array}
\end{equation*}

({\em String paring} \cite[p. 243, Definition VIII.7.2]{CN10})
The string function $\numseq{X,Y}$ is defined by the axiom
\begin{equation*}
\numseq{X_0, X_1} (i) \leftrightarrow
(\exists j \leq i) [(i = \numseq{0,j} \wedge X_0 (j)) \vee
                    (i = \numseq{1,j} \wedge X_1 (j))].
\end{equation*} 
Correspondingly, a pair of strings can be unpaired as
$\numseq{Z_0, Z_1}^{[i]} = Z_i$ ($i=0, 1$).

({\em String constant, string successor, string addition} 
\cite[p. 112, Example V.4.17]{CN10})
The string constant $\emptyset$ is defined by the axiom
$\emptyset (i) \leftrightarrow i < 0$.
The string successor $S(X)$ is defined by the axiom
\begin{equation*}
%\begin{array}{rcll}
S(X) (i) \leftrightarrow i \leq |X|  \wedge 
[X(i) \wedge (\exists j < i) \neg X(j)] \vee 
[\neg X(i) \wedge (\forall j < i) X(j)]. 
%\end{array}
\end{equation*} 
The string addition $X+Y$ is defined by the axiom
\begin{equation*}
(X+Y)(i) \leftrightarrow 
(i < |X| + |Y| \wedge 
 (X(i) \oplus Y(i) \oplus Carry (i, X, Y))
),
\end{equation*}
where $\oplus$ denotes ``exclusive or'', i.e., 
$p \oplus q \equiv
 (p \wedge \neg q) \vee (\neg p \wedge q)$, and
\begin{equation*}
Carry (i, X, Y) \leftrightarrow
(\exists k < i) 
[ X(k) \wedge Y(k) \wedge
 (\forall j < i) (k<j \rightarrow X(j) \vee Y(j)
].
\end{equation*}

({\em String ordering} \cite[p. 219, Definition VIII.3.5]{CN10})
The string relation $X < Y$ is defined by the axiom
\begin{eqnarray*}
X < Y &\leftrightarrow&
%X = Y \vee (
|X| \leq |Y| \wedge \\
&&
(\exists i \leq |Y|)
[(\forall j \leq |Y|)(i < j \wedge X(j) \rightarrow Y(j)) \wedge
 Y(i) \wedge \neg X(i)
].
%).
\end{eqnarray*}
We write $X \leq Y$ to denote %the strict part of $\leq$, i.e.,
$X = Y \vee X < Y$.
%\begin{notation}
In addition, we write $x \minus y$ to denote the {\em limited subtraction}:
$x \minus y = \max \{ 0, x-y \}$, and
$|x|$ to denote the {\em devision} of $x$ by $2$:
$|x| = \lfloor x / 2 \rfloor$.
We will write $x - y = z$ if $x \minus y = z$ and $y \leq x$.
We expand the notion of ``$\Phi$-definable in $T$'' 
(presented on page \pageref{d:definable}) to those functions
 involving the numerical sort in addition to the string sort in an obvious way.
Then it can be shown that both $x \minus y$ and $|x|$ are 
$\BSig{0}$-definable in $\mV{0}$, cf. \cite[p. 60]{CN10}.
Furthermore, though much harder to show, it can be also shown that
a limited form of exponential,
$Exp (x, y) = \min \{ 2^x, y\}$, is 
$\BSig{0}$-definable in $\mV{0}$, cf. \cite[p. 64]{CN10}.
%\end{notation}
It is known that if $\mV{0}$ is augmented by adding a collection of
$\BSig{0}$-defining axioms for numerical and string functions, then the resulting
system is a conservative extension of $\mV{0}$, cf. 
\cite[p. 110, Corolalry V.4.14]{CN10}.
Hence we identify $\mV{0}$ with the system resulting by augmenting 
$\mV{0}$
by adding the $\BSig{0}$-defining axioms for those numerical and string
functions and relations defined above.

Furthermore we work over a slight extension of the vocabulary $\LA2$. 
For a formula $\varphi (i, X)$ let $P_{\varphi} (i, x, X)$ denote a
fresh predicate symbol, where $\varphi$ may contain free variables other than 
$i$ and $X$.
%, and $P_{\varphi}$ contains all the free variables appearing in $\varphi$.
%We write $P_{\varphi, x}^X (i)$ instead of $P_{\varphi} (x, X) (i)$.
We write $\LID$ to denote the vocabulary expanded with the new predicate
$P_{\varphi}$ for each $\varphi$.
\begin{definition}[Extension by fixed point predicates]
\label{d:P}
\normalfont
For a system $T$ over a vocabulary $\mathcal L$ such that
$\LA2 \subseteq \mathcal L$,
$T (\LID)$ denotes the conservative extension of $T$
obtained by augmenting $T$ with the following defining axioms for 
$P_{\varphi}$.
\begin{enumerate}
\item $(\forall i < x) [ P_{\varphi} (i, x, \emptyset)
                          \leftrightarrow i < 0
                       ]$.
\label{d:P:1}
\item $(\forall X \leq x+1) 
       (\forall i < x)
       \left[
        P_{\varphi} (i, x, S(X)) \leftrightarrow 
         \varphi (i, P_{\varphi, x}^X)
       \right]$,
       where $\varphi (i, P_{\varphi, x}^X)$ denotes the result of
      replacing every occurrence of $Y(j)$ in $\varphi (i, Y)$ with 
      $P_{\varphi} (j, x, X) \wedge j < x$.
\label{d:P:2}
\end{enumerate}
\end{definition}
\noindent
Now we introduce an {\em axiom of inductive definitions}.

\begin{definition}[Axiom of Inductive Definitions]
\label{d:ID}
\normalfont
Let $\Phi$ be a class of formulas.
Then the axiom schema $(\Phi \text{-} \mathrm{ID})$ of inductive
definitions denotes (the universal closure of)
%asserts that for any natural $x$ there exist two strings 
%$U$ and $V$ such that $$, $$ fulfilling
the following formula, where $\varphi \in \Phi$.
\begin{equation}
\tag{$\Phi$ \text{-ID}}
 (\exists U, V \leq x+1)
 \left[
 V \neq \emptyset \wedge
 (\forall i < x)
 \left(
 P_{\varphi} (i, x, U+V) \leftrightarrow P_{\varphi} (i, x, U)
 \right)
 \right]
%\label{d:ID:4}
\end{equation}
We write $(\Phi \text{-} \mathrm{IID})$ for
$(\Phi \text{-} \mathrm{ID})$ if additionally the formula
$\varphi \in \Phi$ is {\em inflationary}, i.e., if
$(\forall Y \leq x) (\forall i < x) [Y(i) \rightarrow \varphi (i, Y)]$
holds.
\end{definition} 
For notational convention, we write 
$P^X_{\varphi, x} = Y$
to denote
$(\forall i < x)
 [P_{\varphi} (i, x, X) \leftrightarrow Y (i)]$.
By definition, 
$P_{\varphi, x}^X$ denotes the string consisting of the first $x$ bits of the
string obtained by $X$-fold iteration of the
operator defined by the formula $\varphi$
(starting with the empty string).

\begin{definition}
\normalfont
Let $\Phi$ be a class of $\LA{2}$-formulas.
\begin{enumerate}
\item $\Phi \text{-} \mathrm{ID} :=
       \mV{0} (\LID) + 
       (\Phi \text{-} \mathrm{ID})$.
\item $\Phi \text{-} \mathrm{IID} :=
       \mV{0} (\LID) + 
       (\Phi \text{-} \mathrm{IID})$.
\end{enumerate}
\end{definition}

%\noindent
By definition, the inclusion 
$\Phi \text{-} \mathrm{IID} \subseteq \Phi \text{-} \mathrm{ID}$
holds for any class $\Phi$ of $\LA2$-formulas.
It is important to note that $(\BSig{i} (\LID) \text{-COMP})$ is not allowed in 
$\mV{i} (\LID)$ for any $i \geq 0$, and hence
$\forall x \forall X (\exists Y \leq x) %(\forall i < x)
 P_{\varphi, x}^X = Y$
does not hold in $\mV{0} (\LID)$.
The main theorem in this paper is stated as follows.

\begin{theorem}
\label{t:main}
\begin{enumerate}
\item A function is polytime computable if and only if it is 
      $\BSig{1}$ $(\LID)$-definable in 
      $\BSig{0} \text{-} \mathrm{IID}$.
\item A function is polyspace computable if and only if it is 
      $\BSig{1} (\LID)$-definable in 
      $\BSig{0} \text{-} \mathrm{ID}$.
\end{enumerate}
\end{theorem}

\section{Defining $\mathrm P$ functions by inflationary inductive definitions}
\label{s:P}

\begin{theorem}
\label{t:polytime}
Every polytime function is 
$\BSig{1} (\LID)$-definable in 
$\BSig{0} \text{-} \mathrm{IID}$.
\end{theorem}

\begin{proof}
Suppose that a function $f$ is polytime computable.
Assuming without loss of generality that $f$ is a unary function such that
$f(X)$ can be computed by a single-tape Turing machine $M$
 in a step bounded by a polynomial $p (|X|)$ in the binary length $|X|$
 of an input $X$.

We can assume that each configuration of $M$ on input $X$ is encoded into a
 binary string whose length is exactly $q(|X|)$ for some polynomial $q$.
The polynomial $q$ can be found from information on the polynomial $p$
 since $|f(X)| \leq p(|X|)$ holds.
Let the predicate $\Init_M$ denote the {\em initial} configuration of
 $M$ and $\Next_M$ the {\em next} configuration of $M$. 
More precisely, 
\begin{itemize}
\item $\Init_M (i, X)$ is true if and only if the $i$th bit of
the binary string that encodes the initial configuration of $M$ on input
 $X$ is $1$, and
\label{page:Init}
\item $\Next_M (i, X, Y)$ is true if and only if $Y$ encodes a
      configuration of $M$ on input $X$ and the $i$th bit of the binary string that encodes the successor configuration of $Y$ is $1$.
Note that $\Next_M (i, X, Y)$ never holds if $Y$ does not encode a
      configuration of $M$, or if $Y$ encodes the final configuration of $M$.
\end{itemize}
Careful readers will see that both $\Init$ and $\Next$ can be expressed
 by $\BSig{0}$-formulas.
We define $\Last (j, Y)$, the last $j$ bits of a string $Y$, which is
 also known as the {\em most significant part} of $Y$, by
\[
 \Last (j, Y) (i) \leftrightarrow
 i < j \wedge Y (|Y| \minus j + i).
 \]
Let $\varphi (i, X, Y)$ denote the formula
\[
 i < |Y| + q (|X|) \wedge  
 [Y(i) \vee \Init_M (i, X) \vee
  \Next_M (i \minus |Y|, X, \Last (q(|X|), Y))
 ].
\]
\label{page:phi}
Clearly $\varphi$ is a $\mathrm{\Sigma}^B_0$-formula.

Now reason in $\BSig{0} \text{-IID}$.
It is not difficult to see that
$\varphi (i, X, Y)$ is inflationary with respect to $Y$.
Hence, by the axiom 
$(\mathrm{\Sigma}^B_0 \text{-} \mathrm{IID})$
of $\BSig{0}$ inflationary inductive definitions, 
we can find two strings $U$ and $V$ such that 
$|U|, |V| \leq q(|X|) \cdot (p(|X|) +1)$,
$V \neq \emptyset$ and
$P^{U+V}_{\varphi, q(|X|) \cdot (p(|X|) +1)} = 
 P^{U}_{\varphi, q(|X|) \cdot (p(|X|) +1)}$.
Hence the following $\BSig{1} (\LID)$ formula $\psi_f (X, Y)$ holds.
\begin{equation*}
\begin{array}{rl}
(\exists U, V \leq q(|X|) \cdot (p(|X|) +1)) &
[V \neq \emptyset \wedge %U \neq V \wedge \\
%&
P^{U+V}_{\varphi, q(|X|) \cdot (p(|X|) +1)} = 
P^{U}_{\varphi, q(|X|) \cdot (p(|X|) +1)} \\
&
\wedge
Y = \Output 
(\Last (q (|X|), P^{U}_{\varphi, q(|X|) \cdot (p(|X|) +1)}))],
\end{array}
\end{equation*}
where $\Output (Z)$ denotes the function $\BSig{0}$-definable in 
$\mV{0}$ (depending on the underlying encoding) which extracts the value of the output from $Z$ if $Z$ encodes
 the final configuration of $M$.
By the definition of $\varphi$, 
$\Last (q (|X|), P^{U}_{\varphi, q(|X|) \cdot (p(|X|) +1)})$
encodes the final configuration of $M$,
since in any terminating computation the same configuration does not
 occur more than once.
Hence $\psi_f (X, Y)$ defines the graph $f(X) = Y$ of $f$.
It is easy to see that 
$\forall X \exists Y \psi_f (X, Y)$ also holds.
The uniqueness of $Y$ such that $\psi_{f} (X, Y)$ can be shown accordingly, allowing us to conclude.
\qed
\end{proof}

\section{Reducing inflationary inductive definitions to $\mV{1}$}
\label{s:IID}

In this section we show that every function
$\BSig{1} (\LID)$-definable in the system 
$\BSig{0} \text{-IID}$ of $\BSig{0}$ inflationary inductive definitions
is polytime computable by reducing 
$\BSig{0} \text{-IID}$ to the system $\mV{1}$.

\begin{definition}
A function $\val (x, X)$, which denotes the numerical value of the
 string consisting of
the last $x$ bits of a string $X$, is defined by
\begin{eqnarray*}
\val (x, \emptyset) &=& 0, \text{ or otherwise,} \\
\val (0, X) &=& 0, \\
\val (x+1, X) &=&
  \begin{cases}
  \val (x, X) & \text{if } |X| \leq x, \\
  2 \cdot \val (x, X) & 
  \text{if } x < |X| \ \& \ \neg X((|X| \minus 1) \minus x), \\
  2 \cdot \val (x, X) +1 & 
  \text{if } x < |X| \  \& \  X((|X| \minus 1) \minus x).
  \end{cases}
\end{eqnarray*}
\end{definition}

\begin{lemma}
\label{l:val}
The function $(x, X) \mapsto \val (x, X)$ is $\BDel{1}$-definable
 in $\mV{1}$ if $x \leq |y|$ for some $y$.
More precisely, the relation
$\val (x, X) = z$
can be expressed by a $\BDel{1}$ formula
$\psi_{\val} (x, y, z, X)$ if $x \leq |y|$,
and the sentence
$\forall y (\forall x \leq |y|) \forall X \exists ! z$ 
$\psi_{\val} (x, y, z, X)$
is provable in $\mV{1}$.
\end{lemma}

\begin{proof}
Let $\psi (x, z, X, Y)$ denote the formula expressing that 
$z =0$ if $|X| =0$, or otherwise 
%\item 
      $(Y)^0 = 0$,
%\item 
      $(Y)^x = z$, and for all $j < x$,
\begin{itemize}
%\item $(\forall j < x) 
\item   $|X| \leq j \rightarrow (Y)^{j+1} = (Y)^j$,
%        \wedge
\item   $j < |X| \wedge \neg X (|X| \minus j \minus 1) 
         \rightarrow (Y)^{j+1} = 2 (Y)^j$, and
%        \wedge 
\item   $j < |X| \wedge X (|X| \minus j \minus 1) 
         \rightarrow (Y)^{j+1} = 2 (Y)^j +1$.
\end{itemize}
Define $\psi_{\val} (x, y, z, X, Y)$ to be
$(\exists Y \leq \numseq{x, 2y+1} +1) \psi (x, z, X, Y)$.
Clearly $\psi_{\val}$ is a $\BSig{1}$ formula expressing the
 relation $\val (x, X) = z$ in case $x \leq |y|$. 
Note that $2^{|y|} \leq 2y+1$ for all $y$.
Hence if $x \leq |y|$, then 
$\val (x, X) \leq 2^x \leq 2^{|y|} \leq 2y+1$.
Reason in $\mV{1}$.
One can show that if $x \leq |y|$, then
$(\exists z \leq 2y+1)(\exists Y \leq \numseq{x, 2y+1} +1) 
 \psi (x, z, X, Y)$
holds by induction on $x$.
Accordingly the uniqueness of those $z$ and $Y$ above can be also shown.
From the uniqueness of $z$ and $Y$, 
$\val (x, X) = z$ is equivalent to a $\BPi{1}$ formula
$(\forall u \leq 2y+1)
 (\exists Y \leq \numseq{x, 2y+1} +1) 
 [\psi (x, y, u, X, Y)
  \rightarrow u=z
 ]$.
Hence $\psi_{\val}$ is a $\BDel{1}$ formula.
\qed
\end{proof}

\begin{lemma}
\label{l:IID}
Let $\varphi (x, X)$ be a $\BSig{0}$ formula.
Then the relation
$(x, X, Y) \mapsto P^X_{\varphi, x} =Y$
can be expressed by a $\BDel{1}$ formula 
$\psi_{\Pphi} (x, y, X, Y)$
if $|X| \leq |y|$.
More precisely, corresponding to Definition
\ref{d:P}.\ref{d:P:1} and \ref{d:P}.\ref{d:P:2},
$\psi_{\Pphi}$ enjoys the following.
\begin{enumerate}
%\item $(\forall X \leq x+1)
%       [|X| \leq |y| \rightarrow
%        \forall Y 
%        (\psi_{\Pphi} (x, y, X, Y) \rightarrow |Y| \leq x)
%       ]$.
\item $\psi_{\Pphi} (x, y, \emptyset, \emptyset)$.
\label{l:IID:1}
\item $(\forall X \leq x+1)
       (|X| \leq |y| \rightarrow
        \forall Y, Z
        [\psi_{\Pphi} (x, y, X, Y) \wedge
         \psi_{\Pphi} (x, y, S(X), Z) \rightarrow
         (\forall i < x)(Z(i) \leftrightarrow \varphi (i, Y)) 
        ]
       )$.
\label{l:IID:2}
\end{enumerate}
Furthermore, the sentence
$\forall x, y (\forall X \leq |y|) (\exists ! Y \leq x) 
 \psi_{\Pphi} (x, y, X, Y)$
is provable in $\mV{1}$.
\end{lemma}
\begin{proof}%[of Lemma \ref{l:IID}]
Let $\psi (x, X, Y, Z)$ denote a formula which expresses that
%$Z$ encodes the sequence of strings
%$Z_0, \dots, Z_{z-1}$ such that
\begin{itemize}
%\item $ =  \val (|X|, 2^{|X|}, X)$,
\item $(\forall j \leq \val (|y|, X)) |(Z)^j| \leq x$,
\item $Z^{[0]} = \emptyset$, $Z^{[\val (|y|, X)]} = Y$, and
\item $(\forall j < \val (|y|, X))(\forall i < x)
       [Z^{[j+1]} (i) \leftrightarrow \varphi (i, Z^{[j]})]$.
\end{itemize} 
Define $\psi_{\Pphi} (x, y, X, Y)$ to be
$(\exists Z \leq \numseq{\val (|y|, X), x} +1)
 \psi (x, X, Y, Z)$.
Then, since $\varphi$ is a $\BSig{0}$ formula,
$\psi_{\Pphi}$ is a $\BSig{1}$
formula expressing the relation $P^X_{\varphi, x} = Y$ if 
$|X| \leq |y|$.
Reason in $\mV{1}$.
One can show 
$|X| \leq |y| \rightarrow 
 (\exists Y \leq x) \psi_{\Pphi} (x, y, X, Y, Z)$
by induction on $\val (|y|, X)$.
The uniqueness of such strings $Y$ and $Z$ can be also shown. 
Hence, as in the previous proof,
thanks to the uniqueness of $Y$ and $Z$,
$\psi_{\Pphi}$ is a $\BDel{1}$ formula.
\qed
\end{proof}

\begin{definition}
\label{d:pred}
\begin{enumerate}
\item A string function $\One (y)$, which denotes the string  consisting
      only of $1$ of length $y$, is defined by the axiom
$\One (y) (i) \leftrightarrow i < y$.
\item The {\em string predecessor} $P(X)$ is by the axiom
      \begin{equation*}
%      \begin{array}{rl}
      P(X)(i) \leftrightarrow i < |X| \wedge %&
      [ (X(i) \wedge (\exists j < i) X(j)) \vee
        (\neg X(i) \wedge  (\forall j < i) \neg X(j))
      ]. 
%      \end{array}
      \end{equation*}
\if0
\item (String subtraction) We define the {\em string subtraction}
      $X \minus Y$ by the axiom
      \begin{equation}  
      \begin{array}{rl}
      (X \minus Y) (i) \leftrightarrow &
      (X \leq Y \wedge i <0) \vee
      (Y < X \wedge
        [ (X(i) \wedge \neg Y(i)) \vee \\ 
      &
          (\exists k \leq i)
          ( (\neg X(k) \wedge Y(k)) \wedge
            (\forall j \leq i)
            (k < j \rightarrow \neg X(j) \vee Y(j))
          )
        ]
      )
      \end{array}
      \end{equation} 
      \fi
\end{enumerate}
\end{definition}

\begin{lemma}
\label{l:val_P}
\begin{enumerate}
\item In $\mV{0}$, if $0 < |X|$, then
%the following holds.
$%(\exists Y \leq |X|)
 S (P(X)) = X$ %\wedge
holds.
\label{l:val_P:0}
\item In $\mV{1}$, if $x < |y|$, then the following holds.
%      it holds that
\label{l:val_P:1}
\begin{equation}
      \val (|y|, S (\One (x))) = \val (|y|, \One (x)) +1.
\label{e:pred}
\end{equation}
\item In $\mV{1}$, if $0 < |X| \leq |y|$, then
$\val (|y|, P(X)) + 1 = \val (|y|, X)$
holds.
\label{l:val_P:2}
\end{enumerate}
\end{lemma}

\begin{proof}
1. We reason in $\mV{1}$.
Suppose $0 < |X|$.
Then $X(i)$ holds for some $i < |X|$.
Since the axiom $(\BSig{i} \text{-MIN})$ of minimisation for 
$\BSig{i}$ formulas holds in $\mV{i}$, cf.
\cite[p. 98, Corollary V.1.8]{CN10},
there exists an element $i_0 < |X|$ such that
$X(i_0)$ and $(\forall j < i_0) \neg X(j)$ hold.
Define a string $Y$ with use of $(\BSig{0} \text{-COMP})$ by
\begin{equation}
 |Y| \leq |X| \quad \text{and} \quad
 (\forall i < |X|)
 [Y(i) \leftrightarrow 
  (i_0 < i \wedge X(i))  \vee i < i_0 
 ].
\label{e:Y}
\end{equation}
We show (i) $S(Y) = X$ and (ii) $P(X) = Y$.
It is not difficult to see $|S(Y)| = |X|$ and $|P(X)| = |Y|$.
For (i) suppose $i< |S(X)|$ and $S(X)(i)$.
If $Y(i)$ and $(\exists j < i) \neg Y(j)$ hold, then
$i_0 < i$ and $X(i)$ hold by the definition of $Y$.
If $\neg Y(i)$ and $(\forall j < i) Y(j)$ hold, then
$i=i_0$ holds.
By the choice of $i_0$, $X(i_0)$ and 
$(\forall j < i_0) \neg X(j)$, and hence 
$X(i)$ holds.
The converse inclusion can be shown in the same way.
For (ii) suppose $i < |P(X)|$ and $P(X)(i)$.
If $X(i)$  and $(\exists j < i) X(j)$ hold, then
$X(i)$ and $i_0 < i$ by the choice of $i_0$, and hence $Y(i)$.
If $\neg X(i)$ and $(\forall j < i) \neg X(j)$ hold, then
$i < i_0$, and hence $Y(i)$  holds.
The converse inclusion can be shown in the same way.

2. By Lemma \ref{l:val}, both
$\val (|y|, S(\One (x)))$ and $\val (|y|, \One (x))$ can be defined in $\mV{1}$.
We reason in $\mV{1}$.
Suppose $x \leq |y|$.
Then 
$|\val (x, \One (z))| + 1 \leq |\val (x, S(\One (z)))|$ 
$\leq x + 1 \leq |y|$.
We show that
(\ref{e:pred}) holds
by induction on $x$.
In case $x = 0$, $\One (x) = \emptyset$, and hence
$\val (|y|, S (\One (x))) = \val (|y|, S(\emptyset)) = 1 =
 \val (|y|, \emptyset) +1$.
For the induction step, assume by IH (Induction Hypothesis) that
(\ref{e:pred}) holds.
Then
$\val (|y|, S (\One (x+1))) = 2 \cdot \val (|y|, S (\One (x))) =
 2 \{ \val (|y|, \One (x)) +1 \} = (2 \cdot \val (|y|, \One (x)) +1) +1
 = \val (|y|, \One (x+1)) +1$.

3. We reason in $\mV{1}$.
Suppose $0 < |X| \leq |y|$.
Choose an element $i_0 < X$ as above and define a string $Y$ in the same way as (\ref{e:Y}).
Then $Y = P(X)$ as we showed above.
By the choice of $i_0$, for any $j < |X|$, if $i_0 < j$, then
 $X(j) \leftrightarrow Y(j)$ holds.
Hence it suffices to show that
$\val (|y|, \One (i_0)) + 1 = \val (|y|, S(\One (i_0)))$
holds, but this follows from Lemma \ref{l:val_P}.\ref{l:val_P:1}. 
\qed
\end{proof}

\begin{theorem}
\label{t:IID}
Let $\varphi \in \BSig{0}$.
In $\mV{1}$, if $\varphi$ is inflationary, then there exists a string
 $U$ such that $U \leq \One (|x|)$ and the following holds.
\begin{equation}
\label{e:IID}
%\begin{array}{rl}
 \forall Y, Z %& 
 [
 \psi_{\Pphi} (x, 2x, S(U), Y) \wedge
 \psi_{\Pphi} (x, 2x, U, Z) 
%\\ &
 \rightarrow 
 (\forall i < x) (Y(i) \leftrightarrow Z(i))
 ].
%\end{array}
\end{equation}
%the axiom $(\BSig{0} \text{-} \mathrm{IID})$ of 
%$\BSig{0}$ inflationary
% inductive definitions holds in $\mV{1}$.
%Hence $\BSig{0} \text{-} \mathrm{IID}$ can be regarded as a
% subsystem of $\mV{1}$.
\end{theorem}

\begin{proof}
Let us recall a numerical function $numones (x, X)$ which denotes the
 number of elements of $X$, or equivalently the number of $1$ occurring
 in the string $X$, not exceeding $x$
(See \cite[p. 149]{CN10}).
It can be shown that $numones$ is $\BSig{1}$-definable in $\mV{1}$ 
(See \cite[p. 149]{CN10}).
As we observed in the proof of Lemma \ref{l:val} or Lemma \ref{l:IID}, 
$numones$ is even $\BDel{1}$-definable in $\mV{1}$.

Let $\varphi \in \BSig{0}$.
Reason in $\mV{1}$.
Suppose that $\varphi$ is inflationary, i.e.,
$(\forall Y \leq x) (\forall j < x) [Y(i) \rightarrow \varphi (i, Y)]$ holds.
By contradiction we show the existence of a string $U$ such that 
$U \leq \One (|x|)$ and the condition (\ref{e:IID}) holds.
Since $|S(\One (|x|))| = |x|+1 = |2x|$, 
by Lemma \ref{l:IID}
$(\exists ! Y \leq x) P^X_{\varphi, x} = Y$ 
holds for any $X \leq S(\One (|x|))$.
Hence it suffices to find a string $U$ such that
$U \leq \One (|x|)$ and 
$P_{\varphi, x}^{S(U)} = P_{\varphi, x}^U$.
Assume that such a string $U$ does not exist.
Then for any $X \leq \One (|x|)$ there exists $i < x$ such that
$P^{S(X)}_{\varphi, x} (i)$ but
$\neg P^X_{\varphi, x} (i)$. 
This means that
$numones (x, P^X_{\varphi, x}) <
 numones (x, P^{S(X)}_{\varphi, x})$
holds for any $X \leq \One (|x|)$.
%Let $|X| = |x|$ and $|S(X)| = |x| +1$.

\begin{claim}
%(In $\mV{1}$) 
If $X \leq S(\One (|x|))$, then 
$\val (|x|+1, X) \leq numones (x, P^{X}_{\varphi, x})$
holds.
\end{claim}

%Both $\val$ and $numones$ are $\BDel{1}$-definable in $\mV{1}$,
%the claim is a $\mPi{B}{1}$ statement.
We show the claim by %$(\mPi{B}{1} \text{-IND})$, which is equivalent to 
%$(\BSig{1} \text{-IND})$ in $\mV{1}$, 
induction on $\val (|x|+1, X)$.
The base case that $\val (|x|+1, X) =0$ is clear.
For the induction step, consider the case $\val (|x|+1, X) >0$.
In this case, $0 < |X|$, and hence by Lemma \ref{l:val_P}.\ref{l:val_P:2}
%we can find a string $Y$ such that
%$|Y| \leq |X|$,
%$S(Y) = X$ and
$\val (|x| +1, P(X)) +1 = \val (|x|+1, X)$
holds.
Hence by IH 
$\val (|x|+1, P(X)) \leq numones (x, P^{P(X)}_{\varphi, x})$
holds.
By Lemma \ref{l:val_P}.\ref{l:val_P:0},
$S(P(X)) = X$ holds.
This together with IH yields 
$\val (|x|+1, X) = \val (|x|+1, P(X)) +1  \leq 
 numones (x, P^{X}_{\varphi, x})$
since 
$numones (x, P^{P(X)}_{\varphi, x}) <
 numones (x, P^{S(P(X))}_{\varphi, x}) =
 numones (x, P^{X}_{\varphi, x})$.

By the claim
$\val (|x|+1, S(\One (|x|))) \leq 
numones (x, P^{S(\One (|x|))}_{\varphi, x})$ 
holds.
On the other hand
$x < \val (|x| +1, S(\One (|x|)))$ since 
$|x| < |x|+1 = |S(\One (|x|))|$.
Therefore
$x <  numones (x, P^{S(\One (|x|))}_{\varphi, x})$
holds, but this contradicts the definition of $numones$.
\qed
\end{proof}

\begin{theorem}
\label{t:V}
Suppose $1 \leq i$.
If $\BSig{i} (\LID)$ formula $\psi$ is provable in 
$\BSig{0}  \mathrm{\text{-} IID}$, then
 there exists a $\BSig{i}$ formula $\psi'$ provable in $\mV{1}$ and
 provably equivalent to $\psi$ in $\mV{1} (\LID)$.
\end{theorem}

\begin{proof}%[of Theorem \ref{t:V}]
The theorem can be shown by an induction argument on the length of a formal
$\BSig{0}  \mathrm{\text{-} IID}$-proof resulting in $\psi$.
We only discuss the axiom 
$(\BSig{0} \text{-IID})$ of $\BSig{0}$ inflationary inductive
 definitions and kindly refer details to readers.
%One can see that any instance of $(\BSig{0} \text{-IID})$
%is a $\BSig{2}$ formula (with a free variable $x$).
Let $\varphi$ a $\BSig{0}$ formula.
We reason in $\mV{1}$.
Fix a natural $x$ arbitrarily.
Then, since $S(X) = X + S(\emptyset)$, Theorem \ref{t:IID}
%together with Lemma \ref{l:IID} 
yields two strings $U$ and $V$ such that
$|U|, |V| \leq |x| \leq x+1$, $V = \emptyset$, %$|U+V| = |x|+1 = |2x|$, 
and the following hold.
\[
       \forall Y, Z 
       [
        \psi_{\Pphi} (x, 2x, U+V, Y) \wedge
        \psi_{\Pphi} (x, 2x, U, Z) 
        \rightarrow 
        (\forall i < x) (Y(i) \leftrightarrow Z(i))
       ].
\]
Since $|U|, |U+V| \leq |2x|$, Lemma \ref{l:IID} yields unique two
 strings
$Y_0$ and $Z_0$ such that $|Y_0|, |Z_0| \leq x+1$,
$\psi_{\Pphi} (x, 2x, U+V, Y_0)$
and 
$\psi_{\Pphi} (x, 2x, U, Z_0)$
hold.
Hence, by Lemma \ref{l:IID}.\ref{l:IID:2},
$Z_0( i) \leftrightarrow \varphi (i, Y_0)$
holds for any $i < x$.
This together with Lemma \ref{l:IID}.\ref{l:IID:1} allows us to conclude that
the statement
$%(\exists U, V \leq x +1)
% \left[
% V \neq \emptyset \wedge
 (i < x)
 \left(
 Z_0( i) \leftrightarrow \varphi (i, Y_0)
 \right)
% \right]
$
is provably equivalent to 
$(i < x)
\left(
 P_{\varphi} (i, x, U+V) \leftrightarrow P_{\varphi} (i, x, U)
 \right)$
in $\mV{1} (\LID)$.
\qed
\end{proof}

\begin{corollary}
\label{c:IID}
Every function $\BSig{1} (\LID)$-definable in 
$\BSig{0} \text{-} \mathrm{IID}$ is polytime computable.
\end{corollary}

\begin{proof}
Suppose that a $\BSig{1} (\LID)$ sentence $\psi$ is provable in 
$\BSig{0} \text{-} \mathrm{ID}$.
Then by Theorem \ref{t:V} we can find a $\BSig{1}$ sentence
$\psi'$ provable in $\mV{1}$ and provably equivalent to $\psi$ in 
$\mV{1} (\LID)$.
In particular $\psi$ and $\psi'$ are equivalent under the underlying
 standard interpretation.
Hence every function $\BSig{1} (\LID)$-definable in 
$\BSig{0} \text{-} \mathrm{IID}$ 
is $\BSig{1}$-definable in $\mV{1}$. 
Now employing %a  well known fact that every string function $\BSig{1}$ definable in $\mV{1}$ is polytime computable 
%(cf. \cite[p. 148, Lemma VI.4.5]{CN10}) 
Proposition \ref{p:ptime} enables us
 to conclude.
\qed
\end{proof}

\begin{corollary}
\label{c:P}
A predicate belongs to $\mathrm P$ if and only if it is
$\BDel{1} (\LID)$-definable in
$\BSig{0} \text{-} \mathrm{IID}$.
\end{corollary}

\section{Defining $\mathrm{PSPACE}$ functions by non-inflationary inductive
 definitions}
\label{s:PSPACE}

\begin{theorem}
Every polyspace computable function is 
$\mathrm{\Sigma}^B_1 (\LID)$-definable in \\
$\mathrm{\Sigma}^B_0 \text{-} \mathrm{ID}$.
\end{theorem}

\begin{proof}
The theorem can be shown in a similar manner as Theorem
 \ref{t:polytime}.
Suppose that a function $f$ is polyspace computable. 
As in the proof of Theorem \ref{t:polytime} we can assume that $f$ is a
 unary function such that $f(X)$ can be computed by a single-tape Turing
 machine $M$ using a number of cells bounded by a polynomial 
$p(|X|)$ in $|X|$.
Assuming a standard encoding of configurations of $M$ into binary
 strings, the binary length of every configuration is exactly $q(|X|)$
 for some polynomial $q$.
Let $\Init_M$ denote the predicate defined on page \pageref{page:Init}.
A new predicate $\Next'_M (i, X, Y)$ denotes the successor configuration of $Y$, but in contrast to $\Next_M$,
$\Next'_M (i, X, Y)$ does not change if $Y$ encodes the final
 configuration.
More precisely, if $Y$ encodes the final configuration, then
$(\forall i < q(|X|)) (\Next'_M (i, X, Y) \leftrightarrow Y(i))$
holds.
In contrast to the definition of $\varphi$ on page \pageref{page:phi},
let $\varphi (i, X, Y)$ denote the formula
\[
 i < q(|X|) \wedge 
 [\Init_M (i, X) \vee \Next' (i, X, Y)].
\]
It is not difficult to convince ourselves that $\varphi$ is a $\BSig{0}$ formula.
Hence, reasoning in $\BSig{0} \text{-ID}$,
by the axiom 
$(\mathrm{\Sigma}^B_0 \text{-} \mathrm{ID})$
of $\BSig{0}$ inductive definitions, 
we can find two strings $U$ and $V$ such that 
$|U|, |V| \leq q(|X|) +1$,
$V \neq \emptyset$ and
$P^{U+V}_{\varphi, q(|X|) +1} = 
 P^{U}_{\varphi, q(|X|) +1}$ hold.
Hence the following $\BSig{1} (\LID)$ formula $\psi_f (X, Y)$ holds.
\begin{equation*}
\begin{array}{rl}
(\exists U, V \leq q(|X|) +1) &
[V \neq \emptyset \wedge %U \neq V \wedge %\\
%&
P^{U+V}_{\varphi, q(|X|) +1} = 
P^{U}_{\varphi, q(|X|) +1} \wedge \\
&
Y = \Output (P^{U}_{\varphi, q(|X|) +1})
],
\end{array}
\end{equation*}
where $\Output (Z)$ denotes the extraction function $\BSig{0}$-definable in 
$\mV{0}$ as in the proof of Theorem \ref{t:polytime}.
As we observed, 
$P^{U}_{\varphi, q(|X|) +1}$
encodes the final configuration of $M$.
Hence $\psi_f (X, Y)$ defines the graph $f(X) = Y$ of $f$.
Now it is clear that
$\forall X \exists Y \psi_f (X, Y)$ holds.
The uniqueness of $Y$ follows accordingly, allowing us to conclude.
\qed
\end{proof}

%\comments{Checked until here}
\section{Reducing non-inflationary inductive definitions to $\mW{1}{1}$}
\label{s:ID}

In this section we show that every function
$\BSig{1} (\LID)$-definable in the system 
$\BSig{0} \text{-ID}$ of $\BSig{0}$ inductive definitions
is polyspace computable by reducing 
$\BSig{0} \text{-ID}$ to 
a third order system $\mW{1}{1}$ of bounded arithmetic which was
introduced by A. Skelley in \cite{skelley04}.
The third order vocabulary $\LA{3}$ is defined augmenting the second order vocabulary $\mathcal L^2_A$ with the third
order membership relation $\in_3$.
As in the case of the second order membership, the formula of the form
$Y \in_3 \mathcal X$ is abbreviated as $\mathcal X (Y)$.
Third order elements
$\mathcal X, \mathcal Y, \mathcal Z, \dots$
would denote {\em hyper} strings, i.e., $\mathcal X (Y)$ holds if and only if
the $Y$th bit of $\mathcal X$ is $1$.
Classes $\mBSig{i}$, $\mBPi{i}$ and
$\mBDel{i}$ ($0 \leq i$) are defined in the same
manner as $\BSig{i}$,  $\BPi{i}$ and $\BDel{i}$ but third order quantifiers are taken into
account instead of second order ones.
For instance,
$\mBSig{0} = \bigcup_{0 \leq i}
 \BSig{i} (\mathcal L^3_A)$,
and a $\mBSig{1}$ formula is of the form 
$\exists \mathcal X \psi (\mathcal X)$, 
where no third order quantifier appears in $\psi$.
\if0
A {\em strict} $\mBSig{0}$ consists of a third order
existential quantifier followed by a formula with no third order
quantifiers.
The class of formulas consisting of a single second order universal
quantifier followed by a strict $\mBSig{1}$ formula is
denoted by $\forall^2 \mBSig{1}$.
\fi
For a class $\Phi$ of $\LA{3}$-formulas, the axiom of 
$(\mathrm{\Phi} \text{-} 3 \mathrm{COMP})$ is defined by
\begin{equation}
\tag{$\mathrm{\Phi} \text{-} 3 \mathrm{COMP}$} 
\forall x \exists \mathcal Z (\forall Y \leq x)
[\mathcal Z (Y) \leftrightarrow \varphi (Y)],
\label{p:3comp}
\end{equation}
where $\varphi \in \mathrm{\Phi}$.
The system $\mW{1}{1}$ consists of the basic axioms of second order bounded
arithmetic (B1--B12, L1, L2 and SE, \cite[p. 96]{CN10}), 
$(\mBSig{1} \text{-} \mathrm{IND})$,
$(\mBSig{0} \text{-} \mathrm{COMP})$
and
$\mBSig{0} \text{-} 3 \mathrm{COMP}$.

%$\forall^2 \mBSig{1} \text{-} \mathrm{IND}$

\begin{proposition}[Skelley \cite{skelley04}]
\label{p:pspace}
A function is polyspace computable if and only if it is 
$\mBSig{1}$-definable in $\mW{1}{1}$.
\end{proposition}

\begin{remark}
In the original definition of $\mW{1}{1}$ presented in \cite{skelley04},
the axiom $(\mathrm{IND})$ of induction is allowed only for a class 
$\forall^2 \mBSig{1}$ of formulas, which is 
slightly more restrictive than $\mBSig{1}$.
However it can be shown that every $\mBSig{1}$ formula is
 provably equivalent to a $\forall^2 \mBSig{1}$ formula in 
$\mW{1}{1}$ (See \cite[Theorem 2 and Cororally 3]{skelley04}).
\end{remark}

We show that a stronger form of $\BSig{0}$ inductive
definitions holds in  
$\mW{1}{1}$.

\begin{definition}[Axiom of Relativised Inductive Definitions]
\label{d:RID}
\normalfont
We assume a new predicate symbol $P_{\varphi} (i, x, X, Y)$ instead of 
$P_{\varphi} (i, x, X)$ for each $\varphi$.
We replace Definition \ref{d:P}.\ref{d:P:1} and \ref{d:P}.\ref{d:P:2}
 respectively with the following defining axioms.
\begin{enumerate}
\item $(\forall i < x) [ \Pphi (i, x, \emptyset, Y)
                          \leftrightarrow Y (i)
                       ]$.
\label{d:RID:1}
\item $(\forall X \leq x+1) 
       (\forall i < x)
       \left[
        \Pphi (i, x, S(X), Y) \leftrightarrow 
          \varphi (i, P_{\varphi, x}^X [Y])
       \right]$,
%\\
       where $\varphi (i, P_{\varphi, x}^X [Y])$ denotes the result of
      replacing every occurrence of $X(j)$ in $\varphi (i, X)$ with 
      $\Pphi (j, x, X, Y) \wedge j < x$.
\label{d:RID:2}
\end{enumerate} 
Then a relativised form of the axiom of inductive definitions 
denotes the following statement,
where $\varphi \in \Phi$. 
\begin{equation*}
 (\forall Y \leq x)
 (\exists U, V \leq x+1)
 \left[
 V \neq 0 \wedge 
 (\forall i < x) 
 \left(
 \Pphi (i, x, U+V, Y) \leftrightarrow \Pphi (i, x, U, Y)
 \right)
 \right]
\end{equation*}
\end{definition} 
As in the case of the predicate $\Pphi (i, x, X)$,
we write $P_{\varphi, x}^X [Y] = Z$ instead of
$(\forall i < x) (\Pphi (i, x, X, Y) \leftrightarrow Z(i))$.
Apparently the axiom of relativised inductive definitions implies the
original axiom of inductive definitions.

\begin{definition}
\begin{enumerate}
\item The {\em complementary string} $Y^{\mathrm C}_x$ of a string $Y$ of
 length $x$ is defined by the axiom
 $Y^{\mathrm C}_x (i) \leftrightarrow
  i < x \wedge \neg Y(i)$.
\item The {\em string subtraction} $X \minus Y$ is defined by the axiom
\[
 (X \minus Y) (i) \leftrightarrow
 (X \leq Y \wedge i< 0) \vee
 (Y < X \wedge i < |X| \wedge (X + S(Y^{\mathrm C}_{|X|})) (i)
 ).
\]
\end{enumerate}
\end{definition}

It can be shown that in $\mV{0}$, if $|Y| \leq x$, then
$Y + Y^{\mathrm C}_x = \One (x)$,
and hence
$Y + S(Y^{\mathrm C}_x) = S ( \One (x))$
holds.
Thus one can show that 
$|(X + Y) \minus Y| = |X|$
and, for any $i < |X|$,
$[(X + Y) \minus Y] (i) \leftrightarrow
 [X + S(\One (|X+Y|))] (i) \leftrightarrow X (i)$,
concluding $(X + Y) \minus Y = X$.
%We kindly refer the details to readers.

\begin{lemma}
\label{l:ID}
Let $\varphi (x, X)$ be a $\BSig{0}$ formula.
Then the relation
$(x, y, X, Y, Z) \mapsto 
 P^X_{\varphi, x} [Y] = Z$
can be expressed by a $\mBDel{1}$ formula 
$\psi_{\Pphi} (x, X, Y, Z)$ if $|X|, |Y| \leq y$
in the same sense as in Lemma \ref{l:IID}.
Furthermore the sentence
$\forall x, y (\forall X \leq y) (\forall Y \leq x) (\exists ! Z \leq x) 
 \psi_{\Pphi} (x, y, X, Y, Z)$
is provable in $\mW{1}{1}$.
\end{lemma}

\begin{notation}
%In contrast to the constant $\emptyset$, the third order constant 
%$\emptyset^3$ is defined by
%$\emptyset^3 (X) \leftrightarrow X < \emptyset$.
We define a string function $(\mathcal Z)^X$, which denotes the 
$X$th component of a hyper string $\mathcal Z$,  by the axiom
$(\mathcal Z)^X = Y \leftrightarrow \mathcal Z (\numseq{X, Y})$.
For a hyper string $\mathcal Z$ we write 
$\exists ! \mathcal Z \leq x$ to refer to the uniqueness up to elements
 of length not exceeding $x$, i.e., 
$(\exists ! \mathcal Z \leq x) \psi (\mathcal Z)$
denotes $\exists \mathcal Z \psi (\mathcal Z)$ and additionally,
\begin{equation}
% \wedge
 \forall \mathcal Z_0,  \mathcal Z_1
 [\psi (\mathcal Z_0) \wedge \psi (\mathcal Z_1) \rightarrow
  (\forall Y \leq x) (\mathcal Z_0 (Y) \leftrightarrow \mathcal Z_1 (Y))
 ].
\label{e:!}
\end{equation}
\end{notation}

\begin{proof}%[of Lemma \ref{l:ID}]
Let $\psi (x, y, X, Y, Z, \mathcal Z)$ denote the 
$\mBSig{0}$ formula expressing
\begin{itemize} 
%\item $(\forall U \leq y) 
%       ((\exists V \leq x) (\mathcal Z)^{U} = V \rightarrow U \leq X)$,
%\item $(\forall U \leq y)
%       (U \leq X \rightarrow (\exists V \leq x) (\mathcal Z)^{U} = V
%       )$,
\item $(\forall U \leq y) 
        (U \leq X \rightarrow |(\mathcal Z)^{U}| \leq x)$,
\item $(\mathcal Z)^\emptyset = Y$, $(\mathcal Z)^{X} =Z$, and
\item $(\forall U \leq y) 
       (U < X \rightarrow (\forall i < x) 
        [(\mathcal Z)^{S(U)} (i) \leftrightarrow 
         \varphi (i, (\mathcal Z)^{U})
        ]
       )$.
\end{itemize}
By the definition of $\psi$, the relation 
$P^X_{\varphi, x} [Y] = Z$ is expressed by
the $\mBSig{1}$ formula
$\exists \mathcal Z \psi (x, y, X, Y, Z, \mathcal Z)$
if $|X| \leq y$.
It suffices to show that
$(\forall Y \leq x) (\exists ! Z \leq x) 
 (\exists ! \mathcal Z \leq \numseq{|X|, x})$
$\psi (x, X, Y, Z, \mathcal Z)$
hols in $\mW{1}{1}$.

Reason in $\mW{1}{1}$.
We only show the existence of such a string $Z$ and a hyper string 
$\mathcal Z$.
The uniqueness in the sense of (\ref{e:!}) can be shown accordingly.
By induction on $|X|$ we derive the 
$\mBSig{1}$ formula
$(\forall Y \leq x) (\exists Z \leq x) \exists \mathcal Z 
 \psi (x, X, Y, Z, \mathcal Z)$.
The argument is based on a standard ``divide-and-conquer method''.
%We only show the existence of such $Z$ and $\mathcal Z$.
%The uniqueness can be shown in a similar way. 
In the base case, $|X|=0$, i.e., $X = \emptyset$, and hence the assertion is clear.
The case that $|X| =1$, i.e., $X = S(\emptyset)$, is also clear.
Suppose that $|X| >1$.
Then we can find two strings $X_0$ and $X_1$ such that
$|X_0| = |X_1| = |X| -1$ and $X = X_0 + X_1$.
Fix a string $Y$ so that $|Y| \leq x$.
Then by IH we can find a string $Z_0$ and a hyper string 
$\mathcal Z_0$ such that $|Z_0| \leq x$ and
$\psi (x, X_0, Y, Z_0, \mathcal Z_0)$ hold.
Since $|Z_0| \leq x$, another application of IH yields $Z_1$ and 
$\mathcal Z_1$ such that $|Z_0| \leq x$ and
$\psi (x, X_1, Z_0, Z_1, \mathcal Z_1)$ hold.
Define a hyper string $\mathcal Z$ with use of 
$(\mBSig{0} \text{-3COMP})$ by
\begin{equation}
\begin{array}{rl}
 (\forall U \leq \numseq{|X|, x})
 [\mathcal Z(U) \leftrightarrow &
  (U^{[0]} \leq X_0 \wedge %(\exists V \leq x) 
   (\mathcal Z_0)^{U^{[0]}} = U^{[1]}
  ) \vee \\
&
  (X_0 < U^{[0]} \wedge %(\exists V \leq x) 
   (\mathcal Z_0)^{U^{[0]} \minus X_0} = U^{[1]}
  )
 ].
\end{array}
\label{e:CON}
\end{equation}
Intuitively $\mathcal Z$ denotes the concatenation
$\mathcal Z_0 \CON \mathcal Z_1$, the hyper string $\mathcal Z_0$
 followed by $\mathcal Z_1$.
Then by definition 
$\psi (x, X, Y, Z_1, \mathcal Z)$ holds.
%Now an application of 
%$(\mBSig{1} \text{-} \mathrm{IND})$
%enables us to conclude.
Due to the uniqueness of the string $Z$ and the hyper string 
$\mathcal Z$,  
the $\mBSig{1}$ formula
$\exists \mathcal Z \psi (x, y, X, Y, Z, \mathcal Z)$
is equivalent to the $\mBPi{1}$ formula
$(\forall V \leq x) (\forall \mathcal Z \leq \numseq{|X|, x})
 (\psi (x, X, Y, V, \mathcal Z) \rightarrow V=Z
 )$,
and hence is also a $\mBDel{1}$ formula.
\qed
\end{proof}

\begin{lemma}
\label{l:X+Y}
The following holds in $\mW{1}{1}$.
\[
  \forall x, y
 (\forall X \leq y)(\forall Y \leq y)(\forall Z \leq x)
 (|Y + X| \leq y \rightarrow 
  P^X_{\varphi, x} [P^Y_{\varphi, x} [Z]] = P^{Y+X}_{\varphi, x} [Z]
 ).
\]
\end{lemma}

\begin{proof}
By the previous lemma the relation
$P^X_{\varphi, x} [P^Y_{\varphi, x} [Z]] = P^{Y+X}_{\varphi, x} [Z]$
can be expressed by a $\mBDel{1}$ formula if
$|X|, |Y| \leq y$.
Reason in $\mW{1}{1}$.
We show that
\[
 |X| \leq y \rightarrow 
 (\forall Y \leq y) (\forall Z \leq x) 
 (|Y+X| \leq y \rightarrow 
  P^X_{\varphi, x} [P^Y_{\varphi, x} [Z]] = P^{Y+X}_{\varphi, x} [Z]
 )
\]
holds by induction on $|X|$.
The base case that $|X|=0$ or $|X| = 1$ is clear.
Suppose $|X| > 0$.
Then we can find two strings $X_0$ and $X_1$ such that
$|X_0| = |X_1| = |X| -1$ and
$X_0 + X_1 = X$.
Fix a string $Z$ so that $|Z| \leq x$.
Since
$|X_1| = |X_0| < |X| \leq y$ and
$|X_0 + X_1| = |X| \leq y$,
IH yields 
$P^{X_0}_{\varphi, x} [P^{X_1}_{\varphi, x} [Z]] =
 P^{X_0 + X_1}_{\varphi, x} [Z]$.
Hence
\begin{equation}
 P^{Y}_{\varphi, x} [P^X_{\varphi, x} [Z]] =
 P^{Y}_{\varphi, x} [P^{X_0}_{\varphi, x} [P^{X_1}_{\varphi, x} [Z]]].
\label{e:X+Y:1}
\end{equation}
On the other hand, since
$|X_0| \leq y$,
$|Y+X_0| \leq |Y+X| \leq y$
and
$|P^{X_1}_{\varphi, x} [Z]| \leq x$,
another application of IH yields
\begin{equation}
 P^{Y}_{\varphi, x} [P^{X_0}_{\varphi, x} [P^{X_1}_{\varphi, x} [Z]]] =
 P^{Y+X_0}_{\varphi, x} [P^{X_1}_{\varphi, x} [Z]].
\label{e:X+Y:2}
\end{equation}
Farther, since 
$|Y+X_0| \leq y$ and $|X_1| \leq |X| \leq x$, 
the final application of IH yields
\begin{equation}
P^{Y+X_0}_{\varphi, x} [P^{X_1}_{\varphi, x} [Z]] =
 P^{Y+X_0 + X_1}_{\varphi, x} [Z] =
 P^{Y+X}_{\varphi, x} [Z].
\label{e:X+Y:3}
\end{equation}
Combining equation (\ref{e:X+Y:1}), (\ref{e:X+Y:2}) and (\ref{e:X+Y:3}) allows us to conclude.
\qed
\end{proof}

\begin{definition}
\label{d:numones}
A string function $\numones [Y] (X, \mathcal X)$, which counts the number of elements of
$\mathcal X$ (starting with $Y$) such that $\leq X$, is defined by
\begin{eqnarray*}
\numones [Y] (\emptyset, \mathcal X) &=& Y, \\
\numones [Y] (S(X), \mathcal X) &=& 
  \begin{cases}
  S(\numones [Y] (X, \mathcal X)) & \text{if $\mathcal X (X)$ holds,} \\
  \numones [Y] (X, \mathcal X) & \text{if $\neg \mathcal X (X)$ holds.}
  \end{cases} 
\end{eqnarray*}
\end{definition}

\begin{lemma}
\label{l:numones}
The function $\numones$ is $\mBDel{1}$-definable in
$\mW{1}{1}$.
\end{lemma}

\begin{proof}
Let $\psi_{\numones} (X, Y, Z, \mathcal X, \mathcal Y)$ denote the $\mBSig{0}$
formula expressing
\begin{itemize}
\item $Z \leq Y + X$, 
\item $(\mathcal Y)^\emptyset = Y$, 
\item $(\mathcal Y)^{X} = Z$,
\item $(\forall U \leq |X|)
       [U < X \wedge \mathcal X (U) \rightarrow 
        (\mathcal Y)^{S(U)} = S((\mathcal Y)^U)
       ]$, and
\item $(\forall U \leq |X|)
       [U < X \wedge \neg \mathcal X (U) \rightarrow 
        (\mathcal Y)^{S(U)} = (\mathcal Y)^U
       ]$.
\end{itemize}
Then by definition the $\mBSig{1}$ formula
$\exists \mathcal Y 
 \psi_{\numones} (X, Y, Z, \mathcal X, \mathcal Y)$
defines the graph
$\numones [Y] (X, \mathcal X) = Z$
of $\numones$.
We show that if $|X| \leq x$, then
\[
 (\forall Y \leq x)
 [|Y+X| \leq x \rightarrow 
  (\exists ! Z \leq x) 
  (\exists ! \mathcal Y \leq \numseq{|X|, x})
  \psi_{\numones} (X, Y, Z, \mathcal X, \mathcal Y)
 ]
\]
holds in $\mW{1}{1}$.
Reason in $\mW{1}{1}$.
Given $x$, we only show the existence of such a string $Z$ and a hyper
 string $\mathcal Y$ by induction on $|X|$.
The uniqueness can be shown in a similar manner.
Fix $x$ and $Y$ so that $|Y| \leq x$ and $|Y+X| \leq x$.
In case that $|X| = 0$, i.e., $X = \emptyset$,
define $\mathcal Y$ by
\[
 (\forall U \leq \numseq{0, x})
 [\mathcal Y (U) \leftrightarrow (U = \numseq{\emptyset, Y})].
\] 
Then 
$|Y| \leq x$,
$Y \leq Y + \emptyset$
and
$\psi_{\numones} (\emptyset, Y, Y, \mathcal X, \mathcal Y)$
hold.
In the case that $|X| = 1$, i.e., $X = S(\emptyset)$,
define $\mathcal Y$ by
\begin{equation*}
\begin{array}{rl}
(\forall U \leq \numseq{1, x})
[\mathcal Y (U) 
\leftrightarrow
U = \numseq{\emptyset, Z} \vee %\\
&
(\mathcal X (\emptyset) \wedge U = \numseq{S(\emptyset), S(Y)}) \wedge
   \\   
&
(\neg \mathcal X (\emptyset) \wedge U = \numseq{S(\emptyset), Y})   
].
\end{array}
\end{equation*}
Clearly
$|(\mathcal Y)^{S(\emptyset)}| \leq |S(Y)| =
 |Y + S(\emptyset)|$,
$(\mathcal Y)^{S(\emptyset)} \leq S(Y) = Y + S(\emptyset)$
and
$\psi_{\numones} (S(\emptyset), (\mathcal Y)^{S(\emptyset)}, 
                  Y, \mathcal X, \mathcal Y)$
hold.
For the induction step, suppose $|X| >1$.
Then there exist strings $X_0$ and $X_1 $ such that 
$|X_0| = |X_1| = |X| -1$
and $X_0 + X_1 = X$.
By assumption $|Y + X_0| \leq |Y+X| \leq x$.
Hence IH yields a string $Z_0$ and a hyper string $\mathcal Y_0$ such that
$|Z_0| \leq x$
and
$\psi_{\numones} (X_0, Y_0, Z_0, \mathcal X, \mathcal Y_0)$
hold.
In particular $Z_0 \leq Y + X_0$ holds.
This implies
$|Z_0 + X_1| \leq |Y + X_0 + X_1| = |Y + X| \leq x$.
Thus another application of IH yields
$Z_1$ and $\mathcal Y_1$ such that 
$|Z_1| \leq x$
and
$\psi_{\numones} (X_1, Z_0, Z_1, \mathcal X, \mathcal Y_1)$
hold.
Define $\mathcal Y$ in the same way as (\ref{e:CON}) in the proof of
 Lemma \ref{l:ID}, i.e.,
$\mathcal Y = \mathcal Y_0 \CON \mathcal Y_1$.
It is not difficult to see that
$\psi_{\numones} (X, Y, Z_1, \mathcal X, \mathcal Y)$
holds.
%Letting $y = |X|$ enables us to conclude.
Thanks to the uniqueness of $Z$ and $\mathcal Y$,
one can see that
$\exists \mathcal Y 
 \psi_{\numones} (X, Y, Z, \mathcal X, \mathcal Y)$
is a $\mBDel{1}$ formula.
\qed
\end{proof}

\begin{lemma}
\label{l:comp}
The axiom
$(\mBSig{1} \text{-} 3 \mathrm{COMP})$ 
of third order comprehension for $\mBSig{1}$ formulas holds in $\mW{1}{1}$.
\end{lemma}
Readers might recall that the $\mV{1}$ can be axiomatised by 
$(\BSig{0} \text{-COMP})$ and
$(\BSig{1} \text{-IND})$ instead of
$(\BSig{1} \text{-COMP})$, cf. \cite[p. 149, Lemma VI.4.8]{CN10}.
Lemma \ref{l:comp} can be shown with the same idea as the proof of this fact.
For the sake of completeness, we give a proof in the appendix.
%We start with showing a couple of auxiliary lemmas.

\begin{theorem}
\label{t:ID}
The axiom 
$(\mathrm{\Sigma}^B_0 \text{-} \mathrm{ID})$
of $\BSig{0}$ inductive definitions holds in $\mW{1}{1}$
in the same sense as in Theorem \ref{t:V}.
%Hence $\mathrm{\Sigma}^B_0 \text{-} \mathrm{ID}$ can be regarded as a subsystem of $\mW{1}{1}$.
\end{theorem}

\begin{proof}%[of Theorem \ref{t:ID}]
Instead of showing that the axiom $(\BSig{0} \text{-ID})$ 
holds in $\mW{1}{1}$, we show that even the axiom of relativised
$\BSig{0}$ inductive definitions holds in $\mW{1}{1}$.  
Let $\varphi \in \BSig{0}$.
We reason in $\mW{1}{1}$.
Fix $x$ arbitrarily.
Given $X$ and $Y$, we define a hyper string $\mathcal P^X [Y]$ 
with use of 
$(\mBSig{1} \text{-3COMP})$ by
\[
 (\forall Z \leq x) [\mathcal P^X [Y] (Z) \leftrightarrow 
 (\exists U \leq |X|) [U < X \wedge P^U_{\varphi, x} [Y] = Z]
           ].
\]
\begin{claim}
For a string $W$, if $x < |W|$, then the following holds.
\begin{equation}
\begin{array}{rl}
(\forall Y \leq x)&
[\numones (W, \mathcal P^X [Y]) \leq X \rightarrow \\
&(\exists U, V \leq |X|)
 (U < V \leq X \wedge
  P^V_{\varphi, x} [Y] = P^U_{\varphi, x} [Y] 
 )
].
\end{array}
\label{e:U=V}
\end{equation}
\end{claim} 
Assume the claim. 
Since 
$\numones (S(\One (x)), \mathcal P^{S(\One (x))} [Y]) \leq 
 S(\One (x))$
by the definition of $\numones$ and $x < x+1 = |S(\One (x))|$,
(\ref{e:U=V}) then implies the instance of ($\BSig{0} \text{-ID}$)  
in case of $\varphi$.

The rest of the proof is devoted to prove the claim.
Let us observe that (\ref{e:U=V}) is a $\mBSig{1}$ statement.
We show the claim by induction on $|X|$.
In the base case, $X=\emptyset$ and hence (\ref{e:U=V}) trivially holds.
The case that $X = S(\emptyset)$ is also trivial.
For the induction step, suppose $|X| > 1$.
Then there exist strings $X_0$ and $X_1$ such that
$|X_0| = |X_1| = |X| -1$ and 
$X_0 + X_1 = X$.
Fix a string $Y$ so that $|Y| \leq x$ and suppose 
$\numones (W, \mathcal P^X [Y]) \leq X$.
By the definition of the hyper string $\mathcal P^X [Y]$ and Lemma 
\ref{l:X+Y}, 
for any $Z$, if $|Z| \leq x$, then
$\mathcal P^X [Y] (Z) \leftrightarrow
 \mathcal P^{X_0} [Y] (Z) \vee 
 \mathcal P^{X_1} [P^{X_0}_{\varphi, x} [Y]] (Z)$
holds, i.e.,
$\mathcal P^X [Y] =
 \mathcal P^{X_0} [Y] \cup \mathcal P^{X_1} [P^{X_0}_{\varphi, x} [Y]] 
$.
On the other hand we can assume that 
%if $|Z| \leq x$, then
$(\forall U < X_0)(\forall V < X_1)
% (S(X_0) \leq V \rightarrow 
 P^U_{\varphi, x} [Y] \neq 
 P^{V}_{\varphi, x} [P^{X_0}_{\varphi, x} [Y]]
% )
$ 
holds, i.e.,
$\mathcal P^{X_0} [Y] \cap \mathcal P^{X_1} [P^{X_0}_{\varphi, x} [Y]] 
 = \emptyset$.
This yields
%\begin{claim}
\begin{eqnarray}
&&
 \numones (W, \mathcal P^{X} [Y]) 
\nonumber \\
&=&
 \numones (W, \mathcal P^{X_0} [Y]) +
 \numones (W, \mathcal P^{X_1} [P^{X_0}_{\varphi, x} [Y]]).
\label{e:numones P}
\end{eqnarray} 
%\end{claim}
%One can see that
%$\mathcal P^X [Z] \setminus \mathcal P^{X_0} [Z] =
% \mathcal P^{X_1} [P^{X_0} [Z]]$.
%We can assume $\mathcal P^{X_0}[Z]$

{\sc Case.}
$\numones (W, \mathcal P^{X_0} [Y]) \leq X_0$:
In this case IH yields 
two strings $U_0$ and $V_0$ such that 
$|U_0|, |V_0| \leq |X_0|$,
$U_0 < V_0 \leq X_0$ and 
$P^{V_0}_{\varphi, x} [Z] = P^{U_0}_{\varphi, x} [Z]$.
Since $|X_0| \leq |X|$ and $\leq X_0 \leq X$, 
we can define $U$ and $V$ by $U=U_0$ and $V=V_0$.

{\sc Case.}
$X_0 < \numones (W, \mathcal P^{X_0} [Y])$:
In this case,
$\numones (W, \mathcal P^{X_1} [P^{X_0} [Y]])$ $\leq X_1$ by the equality
(\ref{e:numones P}).
Since $|P^{X_0} [Y]| \leq x$ by definition, another application of IH
yields two strings $U_1$ and $V_1$ such that
$|U_1|, |V_1| \leq |X_1|$,
$U_1 <  V_1 \leq X_1$ and
$P^{V_1}_{\varphi, x} [P^{X_0}_{\varphi, x} [Y]] =
 P^{U_1}_{\varphi, x} [P^{X_0}_{\varphi, x} [Y]]$
hold.
 Define strings $U$ and $V$ by
$U = X_0 + U_1$ and $V = X_0 + V_1$.
Since 
$P^{V}_{\varphi, x} [Y] =
 P^{V_1}_{\varphi, x} [P^{X_0}_{\varphi, x} [Y]]$
and
$P^{U}_{\varphi, x} [Y] =
 P^{U_1}_{\varphi, x} [P^{X_0}_{\varphi, x} [Y]]$
by Lemma \ref{l:X+Y}, now it is easy to check that the assertion (\ref{e:U=V}) holds.
\qed
\end{proof}

\begin{corollary}
Every function $\BSig{1} (\LID)$-definable in 
$\BSig{0} \text{-} \mathrm{ID}$ is polyspace computable.
\end{corollary}

\begin{proof}
Suppose that a $\BSig{1} (\LID)$ formula $\psi$ is provable in 
$\BSig{0} \text{-} \mathrm{ID}$.
Then, as in the proof of Theorem \ref{t:V}, from Lemma \ref{l:ID} and Theorem \ref{t:ID} one can find a
$\mBSig{1}$ formula $\psi'$ provable in $\mW{1}{1}$ and
provably equivalent to $\psi$ in $\mW{1}{1} (\LID)$.
In particular $\psi$ and $\psi'$ are equivalent under the underlying interpretation.
Hence every string function $\BSig{1} (\LID)$-definable in 
$\BSig{0} \text{-} \mathrm{ID}$
is $\mBSig{1}$-definable in $\mW{1}{1}$.
Thus employing %a fact (Skelley \cite[Thoerem 6]{skelley04}) that every
 %string function $\mBSig{1}$ definable in $\mW{1}{1}$ is
 %polyspace computable 
Proposition \ref{p:pspace} enables us
 to conclude.
\qed
\end{proof}

\begin{corollary}
\label{c:PSPACE}
A predicate belongs to {\rm PSPACE} if and only if it is
$\BDel{1} (\LID)$-definable in 
$\BSig{0} \text{-} \mathrm{ID}$.
\end{corollary}

\section{Conclusion}

In this paper we introduced a novel axiom of finitary inductive
definitions over the Cook-Nguyen style second order bounded arithmetic.
We have shown that over a conservative extension $\mV{0} (\LID)$ of $\mV{0}$
by fixed point predicates, $\mathrm{P}$ can be
captured by the axiom of inductive definitions under
$\BSig{0}$-definable inflationary operators whereas
$\mathrm{PSPACE}$ can be captured by the axiom of inductive
definitions under (non-inflationary) $\BSig{0}$-definable operators.
It seems also possible for each $i \geq 0$ to capture the $i$th level of the polynomial
hierarchy by the axiom of inductive definitions under 
$\BSig{i}$-definable inflationary operator, e.g.,
a predicate belongs to $\mathrm{NP}$ if and only if it is
$\BDel{2} (\LID)$-definable in $\BSig{2} \text{-IID}$.
As shown by Y. Gurevich and S. Shelah in \cite{GS86}, over finite structures
the fixed point of a first order definable inflationary operator can be
reduced the least fixed point of a first order definable monotone
operator.
In accordance with this fact,
it is natural to ask whether the axiom $\BSig{0} \text{-IID}$ of
inflationary inductive definitions for $\BSig{0}$-definable operators can be
reduced a suitable axiom of monotone inductive definitions for 
$\BSig{0}$-definable operators.

\subsection*{Acknowledgment}  

The author thanks Shohei Izawa for the initial discussions about this
subject with him.
He also acknowledges valuable discussions with Toshiyasu Arai during
his visit at Chiba University.
He pointed out that the numerical function $\val$ and the string one
$P^X_{\varphi, x}$ ($|X| \leq |y|$) are not only $\BSig{1}$-definable but even
$\BDel{1}$-definable in $\mV{1}$ 
(Lemma \ref{l:val}, \ref{l:IID}) and 
$P^X_{\varphi, x}$ and $\numones$ are even
$\mBDel{1}$-definable in $\mW{1}{1}$ as well 
(Lemma \ref{l:ID}, \ref{l:numones}).
This observation made later arguments easier.
Finally, he would like to thank Arnold Beckmann for his comments about this
subject.
The current formulation of inductive definitions stems from his
suggestion%
\footnote{The previous formulation can be found at 
\href{http://arxiv.org/abs/1306.5559v3}%
{\tt http://arxiv.org/abs/1306.5559v3}.%
}.

%\bibliographystyle{abbr}
%\bibliography{eguchi}

%\input{draft100317.bbl}

%\pagebreak

\appendix

\section{Proving $(\mBSig{1} \text{-3COMP})$ in
$\mW{1}{1}$}

In the appendix we show Lemma \ref{l:comp} which states that the axiom
$(\mBSig{1} \text{-3COMP})$ of third order comprehension
for $\mBSig{1}$ formulas 
(presented on page \pageref{p:3comp}) holds in $\mW{1}{1}$.
We start with showing a couple of auxiliary lemmas.

\begin{lemma}
\label{l:append1}
In $\mW{1}{1}$ for any number $x$, string $X$ and hyper string
 $\mathcal Z$, if $|X| \leq x$ and
$\emptyset < \numones (X, \mathcal Z)$, then
 the following holds.
\begin{eqnarray*}
%\begin{array}{rl}
% \forall x (\forall X \leq x) \forall \mathcal X &
%&&
%  \emptyset < \numones (X, \mathcal Z) \\
% &\rightarrow& %\\
%&
  (\exists Y \leq x)
  (Y < X \wedge 
   S (\numones (Y, \mathcal Z)) = \numones (X, \mathcal Z)
  ).
%\end{array}
\end{eqnarray*}
\end{lemma}

\begin{proof}
Reason in $\mW{1}{1}$.
Fix $x$ and $\mathcal Z$.
We show the following stronger assertion holds by induction on $|X| \leq x$.
\begin{equation*}
\begin{array}{rl}
 (\forall U \leq x) &
  |U+X| \leq x \wedge
  \numones (U, \mathcal Z) < \numones (U+X, \mathcal Z) \rightarrow \\
&
  (\exists Y \leq x)
  (Y < X \wedge 
   S (\numones (U+Y, \mathcal Z)) = \numones (U+X, \mathcal Z)
  ).
\end{array}
\end{equation*}
If $|X|=0$, i.e., $X=\emptyset$, then 
$\numones (U, \mathcal Z) = \numones (U+X, \mathcal Z)$, 
and hence the assertion trivially holds.
In the case $|X|=1$ , i.e., $X = S(\emptyset)$, if 
$\numones (U, \mathcal Z) < \numones (U+S(\emptyset), \mathcal Z)$,
then the assertion is witnessed by $Y=\emptyset$.
For the induction step, suppose $|X| > 1$.
Then there exist two strings $X_0$ and $X_1$ such that
$|X_0| = |X_1| = |X| -1$ 
and
$X_0 + X_1 = X$.
Fix a string $U$ so that $|U| \leq x$ and suppose that
$|U+X| \leq x$ and 
$\numones (U, \mathcal Z) < \numones (U+X, \mathcal Z)$
hold.
Then $|U+X_0| \leq x$.

{\sc Case.}
$\numones (U, \mathcal Z) = \numones (U+ X_0, \mathcal Z)$:
By IH there exists a string $Y < X_1 < X$ such that $|Y| \leq x$ and
$S(\numones (U+Y, \mathcal Z)) =
 \numones (U+X_1, \mathcal Z) =
 \numones (U+X_0+X_1, \mathcal Z) =
 \numones (U+X, \mathcal Z)$.

{\sc Case.}
$\numones (U, \mathcal Z) < \numones (U+ X_0, \mathcal Z)$:
In this case by IH there exists a string $Y_0 < X_0$ such that 
$|Y_0| \leq x$ and 
$S(\numones (U+Y_0, \mathcal Z)) = \numones (U+X_0, \mathcal Z)$
holds.
If 
$\numones (U+X_0, \mathcal Z) = \numones (U+X, \mathcal Z)$,
then the witnessing string $Y$ can be defined to be $Y_0$.
Consider the case
$\numones (U+X_0, \mathcal Z) < \numones (U+X, \mathcal Z)$.
Then another application of IH yields a string $Y_1 < X_1$ such that
$|Y_1| \leq x$ and 
$S (\numones ((U+X_0) + Y_1, \mathcal Z)) =
 \numones ((U+X_0) + X_1, \mathcal Z)$
hold.
Define a string $Y$ by $X_0 + Y_1$.
Then $|Y| \leq |X| \leq x$,
$Y = X_0 + Y_1 < X_0 + X_1 = X$
and 
$S (\numones (U + Y, \mathcal Z)) =
 \numones (U + X, \mathcal Z)$
hold.
\qed
\end{proof}

\begin{lemma}
\label{l:append2}
In $\mW{1}{1}$, for any number $x$, strings $X$, $Z$ and hyper string $\mathcal Z$, 
if $|X| \leq x$ and
$\emptyset < Z \leq \numones (X, \mathcal Z)$,
%if $|X|, |Z| \leq x$, 
then the following holds.
\begin{eqnarray*}
%\begin{array}{rl}
% \forall x (\forall X \leq x) \forall \mathcal X &
%&&
%  \emptyset < Z \leq \numones (X, \mathcal Z) \\
% &\rightarrow& %\\
%&
  (\exists Y \leq x)
  (Y < X \wedge 
   \numones (Y, \mathcal Z) + Z = \numones (X, \mathcal Z)
  ).
%\end{array}
\end{eqnarray*}
\end{lemma}

\begin{proof}
Reason in $\mW{1}{1}$.
Fix $x$ and $\mathcal Z$.
We show the following stronger assertion holds by induction on $|Z|$.
\begin{equation*}
\begin{array}{l}
(\forall X \leq x) (\forall U \leq x) \\
%&&
%(\forall U \leq x) &
\{
 |U+Z| \leq x \wedge 
  \emptyset < Z \leq \numones (X, \mathcal Z)  \rightarrow \\
%&
  (\exists Y \leq x)
  (Y < X \wedge 
   \numones (Y, \mathcal Z) + U+Z = \numones (X, \mathcal Z) +U
  )
\}.
\end{array}
\end{equation*}
If $|Z| =0$, i.e., $Z=\emptyset$, then the assertion trivially holds.
In the case $|Z| =1$, i.e., $Z = S(\emptyset)$, since
$\numones (Y, \mathcal Z) + U + S(\emptyset) = 
 S (\numones (Y, \mathcal Z)) + U$,
the assertion follows from Lemma \ref{l:append1}. 
For the induction step, suppose $|Z| > 1$.
Then, as in the previous proof, there exist strings
$Z_0$ and $Z_1$ such that
$|Z_0| = |Z_1| = |Z| -1$
and
$Z_0 + Z_1 = Z$.
Fix two strings $X$ and $U$ so that $|X|, |U| \leq x$ and 
$|U+Z| \leq x$ and suppose that 
$\emptyset < Z \leq \numones (X, \mathcal Z)$.
Then, since 
$|U+Z_0| \leq |U+Z| \leq x$ and
$\emptyset < Z_0 < Z \leq \numones (X, \mathcal Z)$,
IH yields a string $Y_0 < X$ such that
$|Y_0| \leq x$ and 
$\numones (Y_0, \mathcal Z) +U+Z_0 =
 \numones (X, \mathcal Z) + U$
hold.
Since 
$|Y_0| \leq |X| \leq x$ and
$|U+ Z_0| \leq |U+ Z| \leq x$,
another application of IH yields a string $Y_1 < Y_0 < X$ such that
$|Y_1| \leq x$ and
$\numones (Y_0, \mathcal Z) + (U+ Z_0) + Z_1 =
 \numones (Y_1, \mathcal Z) + U+ Z_0 =
 \numones (X, \mathcal Z) + U$
holds.
Thus the witnessing string $Y$ can be defined to be $Y_0$.
\qed
\end{proof}

\begin{notation}
In contrast to the empty string $\emptyset$,
we write $\emptyset^3$ to denote the {\em empty hyper string}
defined by the axiom
$\emptyset^3 (X) \leftrightarrow |X| < 0$.
\end{notation}

\begin{proof}[of Lemma \ref{l:comp}]
Suppose a $\mBSig{1}$ formula $\varphi (Z)$.
We have to show the existence of a hyper string $\mathcal Y$ such that 
$(\forall Z \leq x) (\mathcal Y (Z) \leftrightarrow \varphi (Z))$
holds.
%Let $\eta (x, \mathcal Y)$ denote the $\mBSig{1}$ formula
%$(\forall Z \leq x) (\mathcal Y (Z) \rightarrow \varphi (Z))$.
Let $\psi (x, U, X, \mathcal Y)$ denote the following formula.
\[
% \exists \mathcal Y 
 (\forall Z \leq x) (\mathcal Y (Z) \rightarrow \varphi (Z)) \wedge
  X = U + \numones (S (\One (x)), \mathcal Y).
\]
By Lemma \ref{l:numones}, $\psi$ is a 
$\mBSig{1}$ formula, and hence so is
$\exists \mathcal Y\psi (x, U, X, \mathcal Y)$.
Reason in $\mW{1}{1}$.
The argument splits into two (main) cases.

\if0
{\sc Case.} 
$\psi (x, \emptyset, S (\One (x)), \mathcal Y)$ holds for some hyper
 string $\mathcal Y$:
In this case %if $\psi (x, \emptyset, S (\One (x)))$ is witnessed by a
 %string $\mathcal Y$, then 
it is easy to see that
$(\forall Z \leq x) (\mathcal Y (Z) \leftrightarrow \varphi (Z))$ 
holds.
\fi

{\sc Case.} 
$(\exists X \leq x+1)
 [X < S (\One (x)) \wedge 
  \exists \mathcal Y \psi (x, \emptyset, X, \mathcal Y) \wedge
  (\forall Y \leq x+1)
  (Y \leq S (\One (x)) \wedge X < Y \rightarrow
   \neg \exists \mathcal Y \psi (x, \emptyset, Y, \mathcal Y) 
  )
 ]$:
Suppose that a string $X_0$ witnesses this case. 
Let
$\psi (x, \emptyset, X_0, \mathcal Y)$.
Then clearly 
$(\forall Z \leq x) (\mathcal Y (Z) \rightarrow \varphi (Z))$
holds.
We show the converse inclusion by contradiction.
Assume that there exists a string $Z_0$ such that
$|Z_0| \leq x$,
$\varphi (Z_0)$ but
$\neg \mathcal Y (Z_0)$.
Define a hyper string $\mathcal Y'$ by
\[
 (\forall Z \leq x) 
 [\mathcal Y' (Z) \leftrightarrow
%  (Z \neq Z_0 \rightarrow \mathcal Y (Z)) \wedge
  (Z = Z_0 \vee \mathcal Y (Z))
 ].
\]
Then
$(\forall Z \leq x) (\mathcal Y' (Z) \rightarrow \varphi (Z))$ 
by definition, and also
$\numones (S(\One (x)), \mathcal Y)$ $= X_0 < S(X_0) =
 \numones (S(\One (x)), \mathcal Y') \leq
 S(\One (x))
$.
But this contradicts the assumption of this case.

{\sc Case.} The previous case fails:
Namely,
$(\forall X \leq x+1)
 [X < S (\One (x)) \wedge 
  \exists \mathcal Y \psi (x, \emptyset, X, \mathcal Y) \rightarrow
  (\exists Y \leq x+1)
  (Y \leq S (\One (x)) \wedge X < Y \wedge
   \exists \mathcal Y \psi (x, \emptyset, Y, \mathcal Y))
 ]
$
holds.
We derive the following $\mBSig{1}$ formula
%$\eta (x, X)$ 
by induction on $|X|$.
\begin{equation}
\begin{array}{rcl}
% \eta (x, X) \equiv
 (\forall U & \leq & x+1)% \\
% &&
 [
 U+X \leq S (\One (x)) \rightarrow \\
 &&
  (\exists Y \leq x+1) 
  (\exists \mathcal Y \psi (x, U, Y, \mathcal Y) \wedge
   U+X \leq Y \leq S (\One (x))
  )
  ].
\end{array}
\label{e:comp}
\end{equation} 
Assume the formula (\ref{e:comp}) holds.
Let $U = \emptyset$ and $X = S(\One (x))$.
Then by (\ref{e:comp}) we can find a string $Y$ and a hyper string 
$\mathcal Y$ such that $|Y| \leq x+1$ and
$\numones (S(\One (x)), \mathcal Y) = Y = S(\One (x))$.
This means that
$(\forall Z \leq x) \mathcal Y (Z)$ holds, and hence in particular
$(\forall Z \leq x) [ \varphi (Z) \rightarrow \mathcal Y(Z)]$
holds.

In the base case, if $|X| =0$, i.e., $X=\emptyset$, then
$\psi (x, U, U, \emptyset^3)$ holds.
This implies $\psi (x, \emptyset, \emptyset, \emptyset^3)$.
Hence by the assumption of this case, we can find a string
$Y$ and a hyper string $\mathcal Y$ such that
$|Y| \leq x+1$,
$Y \leq S (\One (x))$,
$\emptyset < Y$ and
$\psi (x, \emptyset, Y, \mathcal Y)$.
These imply the case $|X| = 1$, i.e.,
$\psi (x, U, U+Y, \mathcal Y)$ and
$U + S(\emptyset) \leq U +Y$.
For the induction step, suppose $|X| > 1$.
Then there exist two strings $X_0$  and $X_1$ such that
$|X_0| = |X_1| = |X| -1$ and
$X_0 + X_1 = X$.
Fix a string $U$ so that $|U + X| \leq x+1$.
Then by IH we can find a string $Y_0$ and a hyper string $\mathcal Y_0$
such that
$|Y_0| \leq x+1$,
$\psi (x, U, Y_0, \mathcal Y_0)$ and
$U + X_0 \leq Y_0$.

{\sc Subcase.}
%First consider the case 
$U + X_0 = Y_0$:
In this subcase, another application of IH yields a string $Y_1$ and a hyper string 
$\mathcal Y_1$ such that
$|Y_1| \leq x+1$,
$\psi (x, Y_0, Y_1, \mathcal Y_1)$ and 
$Y_0 + X_1 \leq Y_1$.
Since 
$U+X = U +X = Y_0 + X_1 \leq Y_1$,
it can be observed that 
$\psi (x, U, Y_1, \mathcal Y_0 \CON \mathcal Y_1)$
holds.

{\sc Subcase.}
%It remains to consider the case 
$U + X_0 < Y_0$:
In this subcase we can assume that 
$Y_0 < U +X$ holds.
Hence by Lemma \ref{l:append2}, we can find a string 
$V < S( \One (x))$ such that
$\numones (V, \mathcal Y_0) = U+ X_0$ holds.
Define a hyper string $\mathcal Y_0 \seg V$  by
\[
 (\forall Z \leq x) 
 [(\mathcal Y_0 \seg V) (Z) \leftrightarrow
  Z < V \wedge \mathcal Y_0 (Z)
 ].
\]
Then
$\numones (S( \One (x)), \mathcal Y_0 \seg V) = U+X_0$
holds by definition. 
Now we can proceed in the same way as the previous subcase but 
we define the witnessing hyper string $\mathcal Y$ 
by $\mathcal Y = (\mathcal Y_0 \seg V) \CON \mathcal Y_1$.
This completes the proof of Lemma \ref{l:comp}.
\qed
\end{proof}

\end{document}